\documentclass{stylefile}          

\usepackage{amsmath,amsfonts,amssymb}
\usepackage{mathtools}

\usepackage{tikz,setspace}
\usetikzlibrary{patterns}

\definecolor{dblue}{rgb}{0.,0.,0.8}
\definecolor{dgreen}{rgb}{0.,0.6,0.}

\renewcommand{\dim}{d}
\newcommand{\tends}{\rightarrow}
\newcommand{\Hd}{\mathbb{H}^\dim}
\newcommand{\abs}[1]{\left|#1\right|}
\newcommand{\tst}{\textstyle}                                      
\newcommand{\m}[1]{\mathcal{#1}}                                   
\newcommand{\eps}{\varepsilon}                                     
\newcommand{\R}{\mathbb R}                                         
\newcommand{\p}{\partial}                                          
\renewcommand{\theta}{{\vartheta}}                                 
\newcommand{\Om}{\Omega}                                      
\newcommand{\oO}{{\overline{\Om}}}                                  
\newcommand{\pO}{{\p \Om}}                                          
\newcommand{\oh}{\frac{1}{2}}                                      

\let\inf\relax \DeclareMathOperator*\inf{\vphantom{p}inf}
\let\sup\relax \DeclareMathOperator*\sup{\vphantom{p}sup}
\let\liminf\relax \DeclareMathOperator*\liminf{\vphantom{p}liminf}
\let\limsup\relax \DeclareMathOperator*\limsup{\vphantom{p}limsup}
\newcommand{\USC}{\mathrm{USC}}
\newcommand{\LSC}{\mathrm{LSC}}

\title{On the notion of boundary conditions in comparison principles for viscosity solutions}
\headlinetitle{Comparison principles for viscosity solutions}

\lastnameone{Jensen}
\firstnameone{Max}
\nameshortone{M.~Jensen}
\addressone{Department of Mathematics, University of Sussex, Brighton BN1 9QF}
\countryone{England}
\emailone{m.jensen@sussex.ac.uk}

\lastnametwo{Smears}
\firstnametwo{Iain}
\nameshorttwo{I.~Smears}
\addresstwo{INRIA Paris, Paris 75589 \& Université Paris-Est, CERMICS (ENPC), 77455 Marne-la-Vallée}
\countrytwo{France}
\emailtwo{iain.smears@inria.fr}

\abstract{We collect examples of boundary-value problems of Dirichlet and Dirichlet--Neu\-mann type which we found instructive when designing and analysing numerical methods for fully nonlinear elliptic partial differential equations. In particular, our model problem is the Monge--Amp\`ere equation, which is treated through its equivalent reformulation as a Hamilton--Jacobi--Bellman equation. Our examples illustrate how the different notions of boundary conditions appearing in the literature may admit different sets of viscosity sub- and supersolutions. We then discuss how these examples relate to the validity of comparison principles for these different notions of boundary conditions.}



\researchsupported{}

\firstpage{1}

\begin{document}


\section{Introduction}

In this short note we collect a small number of examples which we found instructive when designing and analysing numerical methods for fully nonlinear elliptic partial differential equations (PDE). In particular we are interested in the comparison principle between sub- and supersolutions, as used in the convergence proof by Barles and Souganidis~\cite{BS} for the approximation of viscosity solutions by monotone numerical schemes.

Our model problem is the following simple Monge-Amp\`ere equation
\begin{align} \label{MA}
M (D^2 u) = 0, \qquad M (A) := {\tst \oh} f^2 - \det A
\end{align}
on a domain $\Omega \subset \R^\dim$, $\dim\geq 2$ with $f \ge 0$. The problem is complemented with either Dirichlet or mixed Dirichlet--Neumann boundary conditions, as well as the requirement that $u$ be a convex function. In order to conform to the standard framework of degenerate elliptic operators, we consider the following reformulation of \eqref{MA} as a Hamilton--Jacobi--Bellman (HJB) equation~\cite{Krylov87,FJ16}
\begin{align} \label{HJB}
H (D^2 u) = 0, \qquad H (A) := \sup_{B \in \m S_1} (- B : A + f \, \sqrt{\det B}),
\end{align}
where $\m S_1$ is the set of symmetric positive semidefinite matrices in $ \R^{d \times d}$ with trace~equal to~$1$. In particular, it was shown in~\cite{FJ16} that \eqref{MA} (including the convexity constraint) is equivalent to \eqref{HJB} in the sense of viscosity solutions.

\begin{remark}
Observe that the Barles--Souganidis theorem cannot be considered directly for~\eqref{MA} because~\eqref{MA} is only elliptic on the set of convex functions and its test functions are usually assumed to be convex \cite[Definition 1.3.1]{G01}. This is the reason why we shall work with the equivalent formulation ~\eqref{HJB}.
\end{remark}

Comparison principles are central to the theory of viscosity solutions, both for the analysis of well-posedness of the PDE and for the analysis of numerical methods. While conceptually the statement of a comparison principle requires that subsolutions lie below supersolutions, the different formulations of the boundary conditions and the different sets of available test functions raise the question of the validity of the corresponding comparison principle. For instance, the boundary conditions can be imposed in the following variety of ways:
\begin{enumerate}
\item In the classical sense, where the Dirichlet boundary condition is understood point\-wise everywhere on the boundary; this is the setting for the comparison principle of Theorem~3.3 in the User's Guide \cite{UG} by Crandall,~Ishii and Lions.
\item As in the setting of the Barles--Souganidis theorem \cite{BS}, where the Dirichlet boundary condition is relaxed from its classical pointwise sense, and is understood in a generalised sense that allows extensions of the PDE onto the boundary. This notion of the boundary conditions is the subject of section~\ref{sec:BS} below. We remark that in the Barles--Souganidis theorem~\cite{BS}, the comparison principle required for the analysis was stated as an assumption.
\item As in Definition 7.4 of the User's Guide \cite{UG}, where boundary conditions are relaxed similarly to the Barles--Souganidis approach, but semi-continuity of sub- and supersolutions is assumed from the outset and a closure operation is applied to the second-order jets. See also \cite{BP}, where the semi-continuity for sub- and supersolutions of Hamilton-Jacobi equations is imposed, but the closure of the jets is not introduced.
\end{enumerate}
\noindent We also refer the reader to \cite[Definition~7.1]{UG} on the intermediate notion of the boundary condition named therein as the strong viscosity sense.
The sets of sub- and supersolutions are usually chosen within
\begin{enumerate}
\item the spaces $\USC(\oO)$ of bounded upper semi-continuous functions and $\LSC(\oO)$ of bounded lower semi-continuous functions,
\item or within the function space $C(\oO)$ of continuous functions,
\item or, in the classical setting, within the function space $C(\oO)\cap C^2(\Om)$ of twice continuously differentiable functions.
\end{enumerate}
Here, we shall focus our attention on the semi-continuous case because this is the relevant one for the analysis of numerical methods, where only the semi-continuity of upper and lower envelopes of sequences of numerical solutions is known a priori. Nevertheless, it is worth observing that the existence of a comparison principle may well be conditional to further regularity or structure assumptions on the set of sub- and supersolutions. We point to Section 7.C of \cite{UG} for a general discussion of the subject. In this note we focus on the question of whether or not a comparison principle is available for the different notions of the boundary condition, without additional restrictions on the set of sub- and supersolutions. We take as a reference problem the simple Monge-Amp\`ere equation and illustrate with examples how the different types of Dirichlet conditions impose a constraint on sub- and supersolutions. In turn this also informs us how a numerical convergence analysis may be approached.

While we consider in the subsequent text different notions of viscosity sub- and supersolutions, a function $u$ is always said to be a viscosity solution if it is simultaneously a viscosity subsolution and supersolution.

Given a function $v$ we denote its upper semi-continuous envelope by $v^*$ and its lower semi-continuous envelopes by $v_*$, respectively. More precisely, for all $x\in \oO$,
\begin{equation*} 
\begin{aligned}
v^*(x) \coloneqq \sup_{\substack{\{y_n\}_n \subset \oO\\ y_n \tends x}}\limsup_{n\tends\infty} v(y_n), & &&  v_*(x) \coloneqq \inf_{\substack{\{y_n\}_n \subset  \oO\\ y_n \tends x}}\liminf_{n\tends\infty} v(y_n).
\end{aligned}
\end{equation*}

\section{Dirichlet boundary conditions as in the Barles--Souganidis theorem}\label{sec:BS}

Let $\Om$ be a open subset of $\R^\dim$ and consider the model problem~\eqref{HJB} with a homogeneous Dirichlet boundary condition $u=0$ on $\pO$. In line with Definition 1.1 and equations (1.8), (1.9) of~\cite{BS}, we say that a locally bounded function $v$ is a viscosity subsolution of the boundary value problem if
\begin{equation*}
F_*(D^2 \phi(x), v^*(x), x) \leq 0
\end{equation*}
for all $\phi \in C^2(\oO)$ such that $v^* - \phi$ has a local maximum at $x\in \oO$, where $F_*$ denotes the lower semicontinuous envelope of $F$ defined by
\[
F_*(A, w, x) = 
\begin{cases}
H(A) \quad & : x \in \Omega,\\\min \{ H(A), w \} & : x \in \pO.
\end{cases}
\]
Analogously, $v$ is a viscosity supersolution whenever 
\begin{equation}\label{supersol}
F^*(D^2 \phi(x), v_*(x), x) \ge 0
\end{equation}
for all $\phi \in C^2(\oO)$ such that $v_* - \phi$ has a local minimum at $x\in \oO$, where $F^*$ is the upper semicontinuous envelope of $F$ given by
\[
F^*(A, w, x) = 
\begin{cases}
H(A)  \quad & : x \in \Omega,\\
\max \{ H (A), w \} & : x \in \pO.
\end{cases}
\]
We consider in the following example the Monge--Amp\`ere equation on possibly one of the simplest domains with a boundary, namely a $\dim$-dimensional half-space. In particular, let $\Om = \Hd$, with $\dim\geq 2$, where $\Hd = \{x=(x_1,\dots,x_{\dim} ) \in \R^{\dim},\; x_1>0\}$, and consider the problem~\eqref{HJB} with vanishing source term $f=0$, corresponding to the degenerate elliptic case, complemented with homogeneous Dirichlet boundary conditions on $\pO=\{ x=(x_1,\dots,x_{\dim})\in \R^\dim,\; x_1=0\}$. It is clear that the function $u\equiv 0$ is a viscosity solution of the problem in the sense of \cite{BS}. However, we show below that uniqueness of the viscosity solution fails in this example.

\begin{proposition}\label{example1}
Let $\dim \geq 2$ and let $\Om=\Hd$ as above. For a fixed but arbitrary constant $c>0$, let the locally bounded function $v_c$ be defined by $v_c(x)=0$ if $x\in \Om$ and $v_c(x)=-c$ if $x\in \pO$.
Then $v_c$ is a viscosity solution of \eqref{HJB} in the sense of~{\upshape\cite{BS}}.
\end{proposition}

\begin{proof}
It follows from the definition of $v_c$ that $(v_c)^* \equiv 0$ identically in $\oO$, whereas $(v_c)_* = v$ in $\oO$ since $v_c$ is lower semi-continuous. It is thus clear that $v_c$ is a viscosity subsolution of the problem.

We now prove that the function $v_c$ is also a viscosity supersolution and hence a viscosity solution of the problem in the sense of~\cite{BS}; in particular, we must show that $v_c$ is a viscosity supersolution, i.e.~that \eqref{supersol} holds for all $\phi \in C^2(\oO)$ such that $(v_c)_* - \phi$ has a local minimum at $x\in \oO$. It is clear that~\eqref{supersol} is satisfied whenever $x\in \Om$ is an interior point, since $v_* \equiv 0$ in $\Om$. Hence we need only to consider boundary points $x\in\pO$. Suppose now that $\phi \in C^2(\oO)$ is such that $(v_c)_* - \phi$ has a local minimum at $x\in \pO$.  Then, since $\dim\geq 2$, we may take a unit tangent vector $y = (0,y_1,\dots,y_{d-1})$ to the boundary, with $\abs{y}=1$, noting that for any $\eps\in \R$, $x\pm\eps y \in \pO$. Then, we deduce that, for $\eps>0$ sufficiently small,
\begin{equation}\label{divided_differences}
\frac{\phi(x+\eps y)-2\phi(x)+\phi(x-\eps y)}{\eps^2} \leq 0,
\end{equation}
where we have used the fact that $ (v_c)_*(x\pm\eps y) - \phi(x\pm\eps y) \geq (v_c)_*(x)-\phi(x)$ whenever $\eps$ is small enough, and that that $(v_c)_*(x\pm\eps y) = (v_c)_*(x)$ since $(v_c)_* \equiv -c$ on $\pO$. Therefore, taking the limit $\eps\tends 0$, we deduce from \eqref{divided_differences} that the second-order directional derivative $(y \otimes y^{\top}):D^2 \phi(x) \leq 0$. Note that the matrix $B_y \coloneqq y \otimes y^\top$ belongs to the set $\m S_1$ appearing in \eqref{HJB}, since $ B_y$ is positive semi-definite and has trace equal to $\abs{y}^2 =1$ (recall that $y$ was chosen as a unit vector). Therefore, using the definition of $H(D^2 \phi(x))$ from~\eqref{HJB}, we see that $H(D^2\phi(x)) \geq - B_y : D^2 \phi(x) \geq 0 $, and hence 
\[
F^*(D^2 \phi(x),(v_c)_*(x),x)=\max\{H(D^2\phi(x)),(v_c)_*(x)\}\geq 0,
\]
as required by~\eqref{supersol}. Hence $v_c$ is also a viscosity supersolution and thus a viscosity solution of \eqref{HJB}.
\end{proof}

Proposition~\ref{example1} shows that in general, there may be infinitely many viscosity solutions for \eqref{MA} and \eqref{HJB} with Dirichlet boundary conditions understood in the sense of \cite{BS}. {\em Therefore, by the equivalence of \eqref{MA} and \eqref{HJB}, in general there cannot be a comparison principle between sub- and supersolutions for the Monge--Amp\`ere equation when the Dirichlet boundary conditions are understood in the sense of Barles--Souganidis, even on smooth convex domains!} 

\begin{remark}
In Proposition~\ref{example1}, we considered negative perturbations on the boundary, i.e. $v_c(x)=-c$, with $c>0$. For the case of positive perturbations, i.e. $v_c=c$, it is possible to construct test functions showing that the subsolution property does not hold.
\end{remark}

\section{Dirichlet boundary conditions as in the User's Guide}

The definition of viscosity solution is formulated in a different way in the User's Guide~\cite{UG}. There the gradient and Hessians obtained from the test functions define the jets 
\begin{align*}
J^{2,+} u(x) & :=\left\{ (D \phi(x), D^2 \phi(x)) : \phi \in C^2 \text{ and } u - \phi \text{ has local maximum at } x \right\},\\
J^{2,-} u(x) & :=\left\{ (D \phi(x), D^2 \phi(x)) : \phi \in C^2 \text{ and } u - \phi \text{ has local minimum at } x \right\}.
\end{align*}
These jets may no be rich enough to replace the notion of the classical gradient and Hessian in the proof of a comparison principle in \cite{UG}, which is why one considers the closures
\begin{align*}
\overline{J}^{2,+}_\oO u(x) := \bigl\{ & (p,X) \in \R^d \times \m S : \exists \, (x_n, p_n, X_n) \in \oO \times \R \times \m S \text{ so that } \\
& (p_n, X_n) \in J^{2,+} u(x_n) \text{ and } (x_n, u(x_n), p_n, X_n) \to (x, u(x), p, X) \bigr\},\\
\overline{J}^{2,-}_\oO u(x) := \bigl\{ & (p,X) \in \R^d \times \m S : \exists \, (x_n, p_n, X_n) \in \oO \times \R \times \m S \text{ so that } \\
& (p_n, X_n) \in J^{2,-} u(x_n) \text{ and } (x_n, u(x_n), p_n, X_n) \to (x, u(x), p, X) \bigr\},
\end{align*}
which `inherit' nearby gradients and Hessians. 

In line with Example 1.11, Definition 7.4 and equation (7.24) of~\cite{UG}, we keep the above definitions of $F$, $F_*$ and $F^*$. We say that a function $v$ is a viscosity subsolution of the boundary value problem if $u$ is upper semi-continuous on $\oO$ and 
\[
F_*(A, v(x), x) \le 0 \qquad \forall \; (A, p) \in \overline{J}_{\overline{\Omega}}^{2,+}v(x).
\]
Similarly $v$ is a viscosity supersolution whenever $v$ is lower semi-continuous on $\oO$ and
\[
F^*(A, v(x), x) \ge 0 \qquad \forall \; (A, p) \in \overline{J}_{\overline{\Omega}}^{2,-}v(x).
\]
Consequently, there are two differences with the Barles--Souganidis definition:
\begin{enumerate}[(a)]
\item The equation is tested with a larger set of `derivatives' as a result of the closure of the semi-jets.
\item Both $u$ and $v$ are assumed to be semi-continuous, rather than taking their lower and upper semi-continuous envelopes.
\end{enumerate}
The functions $v_c$ from Proposition~\ref{example1}, which are lower semi-continuous by definition, are not affected by the closure of the jets (a) in the sense that the above arguments from the previous section related to the supersolution property of $v_c$ remain valid without change. 

However, the requirement of semi-continuity (b) means that now, the functions $v_c$ do not qualify as subsolutions, (and thus are not viscosity solutions) in the sense of \cite{UG}. Nevertheless, since $u\equiv 0$ is a viscosity solution and hence is also a subsolution and yet $u\geq v_c$ for all $c>0$, we have found a subsolution that does not lie below the supersolution $v_c$; thus there is again no comparison principle between semi-continuous sub- and supersolutions. We note that there is no contradiction between our example and \cite[Theorem~7.9]{UG}, which asserts only a comparison principle between \emph{continuous} sub- and supersolutions. However, recall that the case of semi-continuous sub- and supersolutions is the relevant one for the study of numerical approximations.

\section{Dirichlet boundary conditions in the classical sense}

\begin{figure}[t]
\includegraphics[width=12cm]{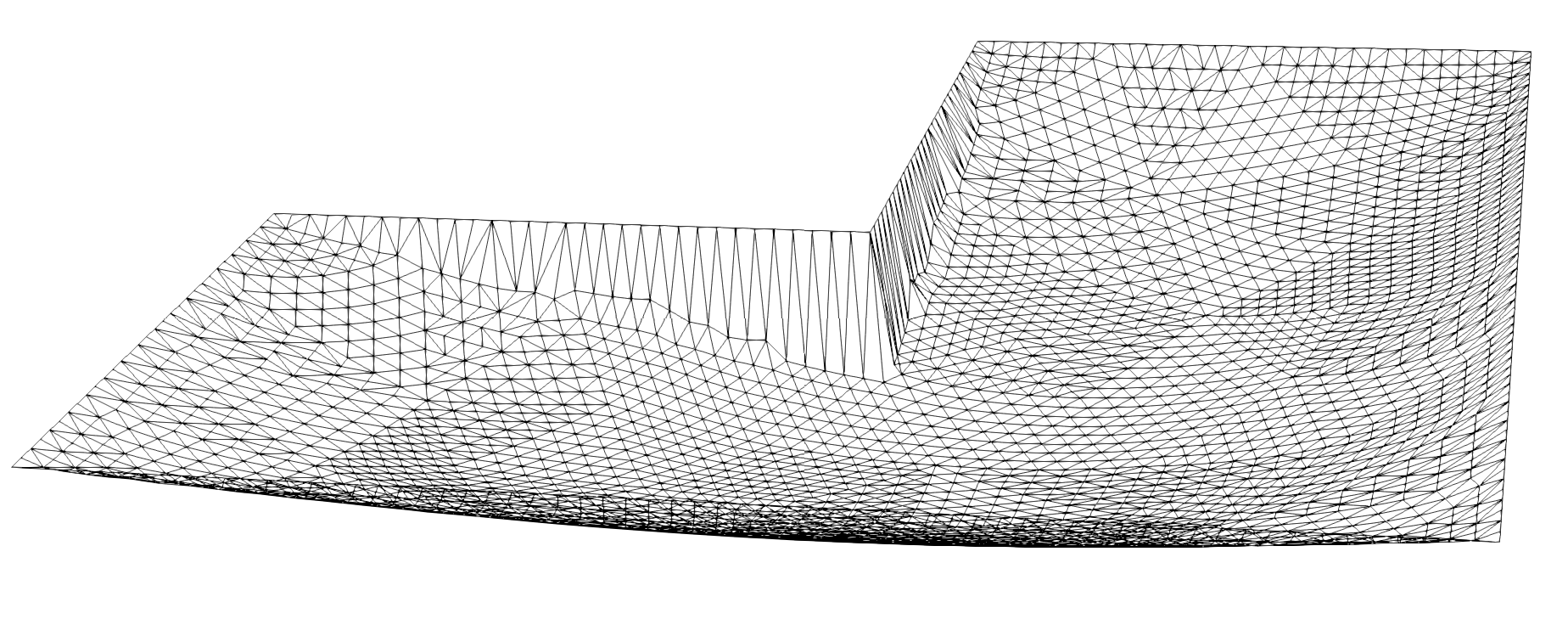}
\caption{Numerical solution of \eqref{HJB} on an L-shaped domain with homogeneous boundary conditions and $f \equiv 1$.}
\label{wireframe}
\end{figure}

As in \cite[Definition 2.2]{UG} we now say that a function $u$ is called a viscosity subsolution (resp.~supersolution) if $u \in \USC(\Omega)$ (resp.~$u\in \LSC(\Omega)$) and if for all $\varphi \in C^2(\Omega)$ such that $u-\varphi$ has a local maximum (resp.~minimum) at $x\in \Omega$ we have
\[
F(D^2\varphi(x),\nabla \varphi(x), u(x), x) \leq 0 
\]
(resp.~$F(D^2\varphi(x),\nabla \varphi(x), u(x), x) \geq 0$). 

\begin{figure}[t]
\centerline{\includegraphics[width=12cm]{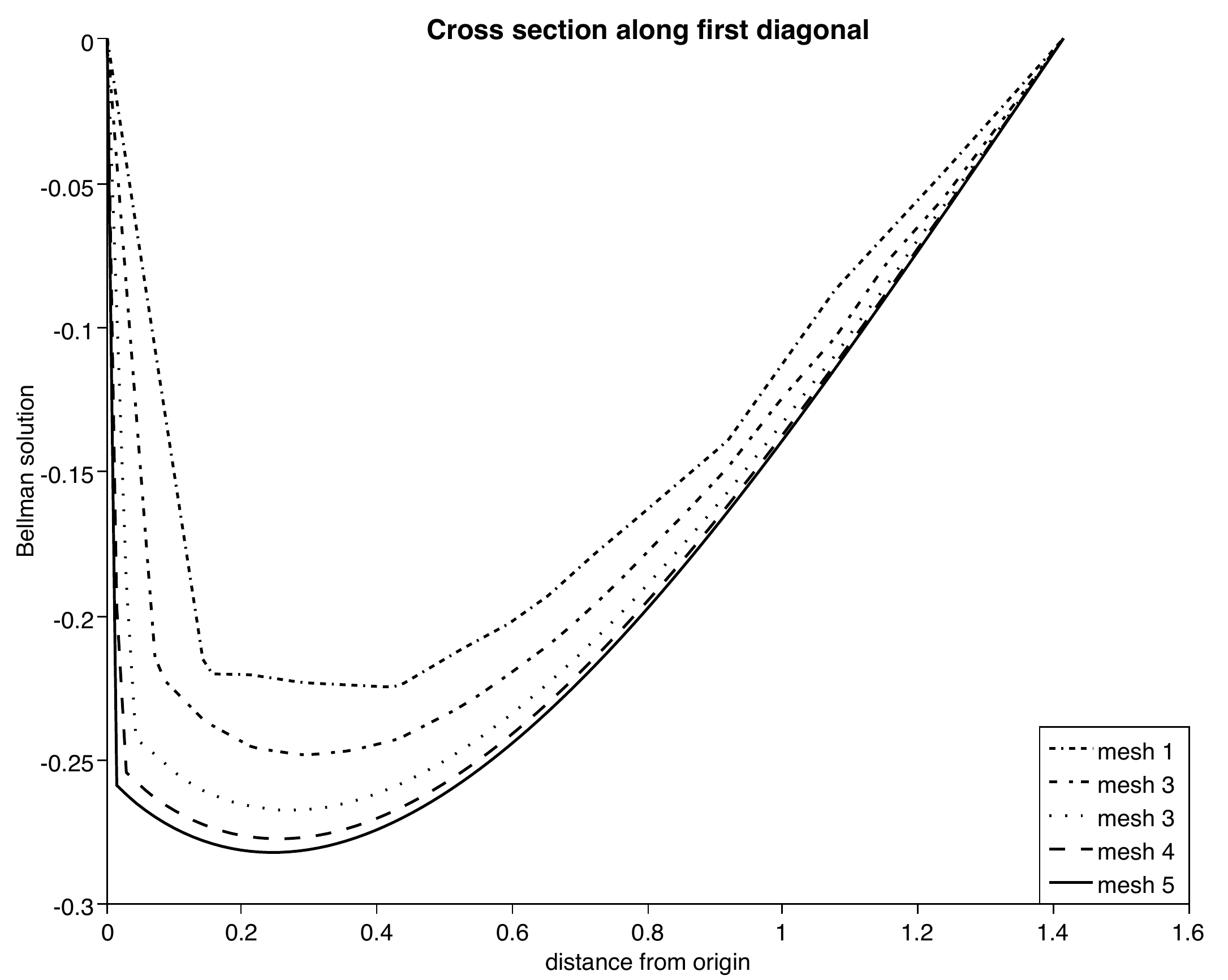}}
\caption{Cross sections of the numerical solution along the first diagonal $x_1 = x_2$.}
\label{cross_section}
\end{figure}

Lemma~8 in \cite{FJ16}, in the spirit of \cite[Section 5.C]{UG}, states that if $u$ is a subsolution and $v$ is  a supersolution of \eqref{HJB} \emph{and crucially if $u \le v$ on $\partial \Omega$}, then $u \le v$ on~$\oO$. Hence a viscosity solution that satisfies the boundary conditions in a pointwise sense is necessarily unique, if it exists.
The general setting of paper \cite{FJ16} is that of a bounded strictly convex domain $\Omega$; however, neither boundedness nor convexity are used in the proof of Lemma~8 of \cite{FJ16}.
The existence and uniqueness of viscosity solutions holds with classical boundary conditions on strictly convex domains. Yet, on non-convex domains the Monge--Amp\`ere problem is in general not well-posed.
Since the Barles--Souganidis theorem on the convergence of numerical approximations is also a proof of the existence of a unique viscosity solution, theorems of this type are therefore bound to fail for \eqref{HJB} on general non-convex domains. It is interesting to pinpoint the step at which the argument breaks down. Lemma 6.4 of \cite{FJ16} shows how the upper and lower semi-continuous envelopes of the numerical solutions in the small-mesh limit satisfy the classical boundary conditions; this argument relies on the existence of certain test functions, for which the strict convexity of the domain is needed.

We shall therefore consider the scheme of \cite{FJ16} for \eqref{HJB} on the L-shape domain
\[
\Omega = \bigl[ (0,1) \times (-1,1) \bigr] \cup \bigl[ (-1,1) \times (0,1) \bigr],
\]
noting that the existence and uniqueness of numerical solutions also holds on non-convex domains. A numerical solution is depicted in Figure \ref{wireframe} while Figure \ref{cross_section} shows the cross sections on $\{ (x_1, x_2) \in \Omega : x_1 = x_2 \}$ of the numerical solutions over several levels of refinement, where mesh 1 is the coarsest with 328 degrees of freedom while mesh~5 has 83968 DoFs. The figures illustrate how a mesh-dependent boundary layer appears in the vicinity of the re-entrant corner. Thus it is reasonable to expect that the lower semi-continuous envelope 
\[
\underline{u}(x) := \liminf_{\substack{y\to x\\ h \to 0}} u_h(y), \quad \forall\,x\in \oO,
\]
of the sequence $(u_h)_h$ of numerical solutions will not satisfy the boundary conditions in the classical sense, so that the above mentioned comparison principle may not be used to guarantee existence of the viscosity solution.

\section{Mixed Dirichlet--Neumann boundary conditions as in the Barles--Souganidis theorem }

We now show some generalisations of the example of section~\ref{sec:BS} to problems with mixed boundary conditions on bounded convex domains in order to highlight some further subtleties and challenges of treating the boundary conditions in a generalised sense. We therefore return to the definition of viscosity sub- and supersolutions of~\cite{BS}, as detailed in section~\ref{sec:BS}. 

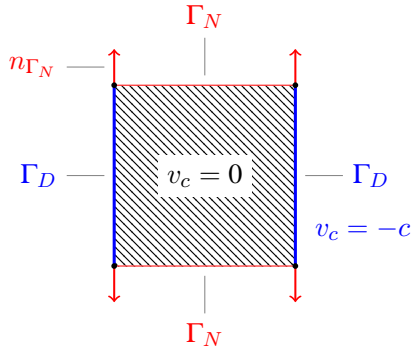
\begin{figure}[t]
\begin{center}
\begin{tikzpicture}[scale=2.5,st1/.style={circle,draw=black,fill=black,thick, minimum size=0.5mm, inner sep=0pt}]
\draw[white,pattern=north west lines,pattern color=black] (0,0) rectangle   (1,1);
\draw[red] (0,0) rectangle  (1,1);
\draw[blue,very thick] (0,0) -- (0,1);
\draw[blue,very thick] (1,0) -- (1,1);
\node[fill=white] at (0.5,0.5) {$v_c=0$};
\draw[red,->,thick] (0,0) node[st1] {} -- (0,-0.2);
\draw[red,->,thick] (1,0) node[st1] {} -- (1,-0.2);
\draw[red,->,thick] (0,1) node[st1] {} -- (0,1.2);
\draw[red,->,thick] (1,1) node[st1] {} -- (1,1.2);
\node[pin=left:{\color{red}$n_{\Gamma_N}$}] at (0,1.1) {};
\node[label=right:{{\color{blue}$v_c=-c$}}] at (1,0.2) {};
\node[pin=left:{{\color{blue}$\Gamma_D$}}] at (0,0.5) {};
\node[pin=right:{{\color{blue}$\Gamma_D$}}] at (1,0.5) {};
\node[pin=below:{{\color{red}$\Gamma_N$}}] at (0.5,0) {};
\node[pin=above:{{\color{red}$\Gamma_N$}}] at (0.5,1) {};
\end{tikzpicture}
\end{center}	
\caption{Construction of the viscosity solutions $v_c$ in~Proposition~\ref{example2}.} 
\end{figure}

Consider the unit square domain $\Om = (0,1)^2$ in two space dimensions, and consider the simple Monge--Amp\`ere equation \eqref{HJB} with mixed Dirichlet--Neumann boundary conditions
\begin{equation}\label{mixed}
\begin{aligned}
H(D^2 u) = 0  & 	& &\text{in }\Om, \\
u = 0 & & & \text{on } \Gamma_{D},\\
\nabla u\cdot n =0 & && \text{on } \Gamma_{N},
\end{aligned}
\end{equation}
where $H(\cdot)$ is as in \eqref{HJB}, where $\Gamma_D = \{x=(x_1,x_2)\in \pO,\; x_1\in\{0,1\}, x_2\in(0,1)\}$ is composed of the left and right faces of $\pO$ (which are open relative to $\pO$), and $\Gamma_N = \{x=(x_1,x_2)\in \pO,\; x_1\in (0,1), x_2\in \{0,1\}\}$ is composed of the top and bottom open faces of $\pO$. Furthermore we introduce $\overline{\Gamma_D}$ the closure  of $\Gamma_D$, and we note that $\overline{\Gamma_D}$ and $\Gamma_N$ partition $\pO$.
To formalize the definition of the viscosity sub- and super-solutions, we define the operator $B\colon \R^{\dim}\times \R \times \pO \tends \R $ by
\begin{equation*}
B(p, r, x) \coloneqq
\begin{cases}
r = 0  & \text{if } x\in \overline{\Gamma_{D}},\\
p\cdot n_{\Gamma_N} =0 & \text{if } x\in \Gamma_{N},
\end{cases}
\end{equation*}
where $n_{\Gamma_N}$ is the unit outward normal on $\Gamma_N$, which in this example is simply given by $n_{\Gamma_N}=(0,1)$ when $x_2=1$, and $n_{\Gamma_N}(0,-1)$ when $x_2=0$.
The lower and upper envelopes of $B$ are given by
\[
B_*(p,r,x) \coloneqq
\begin{cases}
	B(p,r,x) & : x\in \Gamma_D\cup\Gamma_N,\\
	\min\{r, p\cdot n_{\Gamma_N} \} & : x\in \pO \setminus (\Gamma_D\cup \Gamma_N)
\end{cases}
\]
and
\[
B^*(p,r,x) \coloneqq
\begin{cases}
	B(p,r,x) & : x\in \Gamma_D\cup\Gamma_N,\\
	\max\{r, p\cdot n_{\Gamma_N} \} & : x\in \pO \setminus (\Gamma_D\cup \Gamma_N).
\end{cases}
\]
Following \cite{BS} and \cite[Section~7.B]{UG}, a locally bounded function $v$ is called a viscosity subsolution of the boundary value problem \eqref{mixed} if
\begin{equation*}
F_*(D^2 \phi(x), \nabla \phi(x), v^*(x), x) \leq 0
\end{equation*}
for all $\phi \in C^2(\oO)$ such that $v^* - \phi$ has a local maximum at $x\in \oO$, where $F_*$ is defined by
\[
F_*(A, p, w, x) = 
\begin{cases}
H(A) \quad & : x \in \Omega,\\\min \{ H(A), B_*(p,w,x) \} & : x \in \pO.
\end{cases}
\]
Analogously, $v$ is a viscosity supersolution of \eqref{mixed} whenever 
\begin{equation}\label{mixed_supersol}
F^*(D^2 \phi(x), v_*(x), x) \ge 0
\end{equation}
for all $\phi \in C^2(\oO)$ such that $v_* - \phi$ has a local minimum at $x\in \oO$, where $F^*$ is given by
\[
F^*(A, w, x) = 
\begin{cases}
H(A)  \quad & : x \in \Omega,\\
\max \{ H (A), B^*(p,w,x) \} & : x \in \pO.
\end{cases}
\]

It is clear that the function $u\equiv 0$ is a viscosity solution of the boundary value problem \eqref{mixed}. However, we show in Proposition~\ref{example2} below that again uniqueness of the viscosity solution fails due to the lack of a comparison principle.

\begin{proposition}\label{example2}
For a fixed but arbitrary constant $c>0$, let the locally bounded function $v_c$ be defined by $v_c=0$ on $\Om \cup \Gamma_N$ and $v_c=-c$ on $\overline{\Gamma_D}$. Then $v_c$ is a viscosity solution of \eqref{mixed}.	
\end{proposition}
\begin{proof}
The upper envelope $(v_c)^* \equiv 0$ in $\oO$, so we see that $v_c$ is a subsolution. To show the supersolution property, consider a function $\phi \in C^2(\oO)$ such that $(v_c)_* - \phi$ has a local minimum at $x\in \oO$. First, it is clear that \eqref{mixed_supersol} holds for whenever $x\in \Om$ is an interior point or when $x \in \Gamma_N$ is a `Neumann' boundary point. It remains only to consider `Dirichlet' points $x\in \Gamma_D$ and corner points $x\in \pO \setminus (\Gamma_N\cup \Gamma_D)$.

If $x\in \Gamma_D$ is a `Dirichlet' point, i.e. $x=(x_1,x_2)$ with $x_1\in \{0,1\}$ and $x_2\in (0,1)$, then we can follow the same argument used in the proof of Proposition~\ref{example1} to deduce that~$\p_{x_2 x_2}^2 \phi(x) \leq 0$ and hence that $H(D^2 \phi(x))\geq 0$. This implies that \eqref{mixed_supersol} holds whenever $x\in \Gamma_D$.  

The only remaining case is when $x$ is a corner point, i.e.\ $x=\pO\setminus (\Gamma_N\cup \Gamma_D)$.
For this case, we note that for $\eps>0$ sufficiently small, $x-\eps n_{\Gamma_N} \in \Gamma_D$ since $n_{\Gamma_N}=\pm (0, 1)$ is the outward normal for the `Neumann' part of the boundary. Therefore, we deduce that, for all $\eps>0$ sufficiently small,
\begin{equation}\label{divided_differences_2}
\frac{\phi(x)-\phi(x-\eps n_{\Gamma_N})}{\eps} \geq 0,
\end{equation}
where we have used the facts that $(v_c)_*(x-\eps n_{\Gamma_N})- \phi(x-\eps n_{\Gamma_N}) \geq (v_c)_*(x) - \phi(x)$ for $\eps>0$ sufficiently small and that that $(v_c)_*(x-\eps n_{\Gamma_N}) = v_c(x) = -c$.
Therefore, taking the limit $\eps\tends 0$ in \eqref{divided_differences_2} gives $\nabla \phi(x) \cdot n_{\Gamma_N} \geq 0$, and hence $B^*(\nabla \phi(x),(v_c)_*(x),x) = \max\{ \nabla \phi(x) \cdot n_{\Gamma_N}, (v_c)_*(x)\} \geq 0$. Thus we find that \eqref{mixed_supersol} is satisfied in the case where $x$ is a corner point. Hence $v_c$ is also a viscosity supersolution and thus a viscosity solution of \eqref{mixed}.
\end{proof}

The conclusion from Proposition~\ref{example2} is that uniqueness again fails when considering the boundary conditions in the generalised sense as in \cite{BS}.

\end{document}